\documentclass[10pt,english]{smfart}
\usepackage{bbm}
\usepackage[T1]{fontenc}
\usepackage[french,english]{babel}
\usepackage{latexsym,amscd,color}
\usepackage{amsmath,amsfonts,amssymb,mathrsfs}
\usepackage{enumerate,euscript}
\usepackage{amssymb,url,xspace,smfthm}
\usepackage{hyperref}
\input xy
\xyoption{all}

\newcommand{\BibTeX}{{\scshape Bib}\kern-.08em\TeX}

\newcommand{\T}{\S\kern .15em\relax }
\newcommand{\AMS}{$\mathcal{A}$\kern-.1667em\lower.5ex\hbox
        {$\mathcal{M}$}\kern-.125em$\mathcal{S}$}

\DeclareMathOperator{\im}{Im}
\DeclareMathOperator{\vol}{vol}
\DeclareMathOperator{\proj}{Proj}

\DeclareMathOperator{\hn}{HN}
\DeclareMathOperator{\an}{an}

\DeclareMathOperator{\supp}{supp}
\DeclareMathOperator{\asy}{asy}

\DeclareMathOperator{\Vect}{Vect}
\DeclareMathOperator{\Pic}{Pic}

\renewcommand{\sp}{\mathrm{sp}}

\DeclareMathOperator{\spec}{Spec}

\renewcommand{\P}{\mathbb{P}}
\newcommand{\wmu}{\widehat{\mu}}

\newcommand{\adeg}{\widehat{\deg}}

\newcommand{\E}{\overline{E}}

\renewcommand{\O}{\mathcal{O}}
\renewcommand{\L}{\overline{L}}

\newcommand{\ndot}{\raisebox{.4ex}{.}}

\title{Arithmetic Fujita approximation over adelic curves}
\alttitle{L'approximation arithm\'etique de Fujita sur courbes ad\'eliques}

\date{\today}
\author{Chunhui Liu}
\address{Institute for Advanced Study in Mathematics\\
Harbin Institute of Technology\\
150001 Harbin\\P. R. China}
\email{chunhui.liu@hit.edu.cn}

\begin{document}
\def\smfbyname{}
\begin{abstract}
In this paper, we will prove an analogue of Fujita's approximation under the framework of Arakelov theory over adelic curves, which gives a positive response to a conjecture of Huayi Chen and Atsushi Moriwaki.
\end{abstract}
\begin{altabstract}
Dans cet article, on d\'emontrera une analogie de l'approximation de Fujita sous le cadre de la th\'eorie d'Arakelov sur courbes ad\'eliques, qui donne une r\'eponse positive d'une conjecture de Huayi Chen et Atsushi Moriwaki.
\end{altabstract}

\maketitle

\tableofcontents
\section{Introduction}
Fujita approximation is an approximative version of Zariski decomposition of pseudo-effective divisors, and see \cite{Zariski62} for more details. More precisely, let $X$ be a projective variety over a field $k$, and $L$ be a line bundle on $X$. The volume of $L$ is defined as
\[\vol(L)=\limsup_{D\rightarrow\infty}\frac{\dim_kH^0(X,L^{\otimes D})}{D^{\dim(X)}/\dim(X)!}.\]
We say that $L$ is a \textit{big} line bundle if $\vol(L)>0$. Fujita's approximation theorem states that, for all $\epsilon>0$, there exist a $p\in\mathbb N^+$ and a birational morphism $v:X'\rightarrow X$, such that $v^*(L^{\otimes p})$ can be decomposed as the tensor product of an ample line bundle $A$ and an effective line bundle $M$, which satisfy
 \[p^{-\dim(X)}\vol(A)\geqslant\vol(L)-\epsilon.\]
Intuitively, this result asserts that a big line bundle can be arbitrarily closely approximated by ample line bundles along an effective direction.

This approximated decomposition theorem is proved by T. Fujita \cite{Fujita94} for the case of characteristic $0$, and by S. Takagi \cite{Takagi07} for an arbitrary characteristic. We refer the readers to \cite[\S 11.4]{LazarsfeldII} for a survey on this topic.

\subsection{Arithmetic analogue over number fields}
It is quite nature to ask whether there is an analogue of Fujita's approximation theorem under the framework of Arakelov geometry, since we can define the arithmetic positivity and arithmetic volume of an adelic line bundle over arithmetic varieties.
\subsubsection{}
The arithmetic analogue of volume function and arithmetic bigness under the framework of Arakelov geometry are introduced in \cite{Moriwaki00,Moriwaki07} by A. Moriwaki. Soon after these, voluminous works appeared on this topic, such as \cite{Moriwaki08,Chen11,Yuan2015,Ikoma2015,Moriwaki2016,Ikoma2021}.

Let $K$ be a number field, $\mathscr X$ be an arithmetic variety over $\spec \O_K$ of dimension $d$, and $\overline{\mathscr L}$ be a Hermitian line bundle on $\mathscr X$. For a $D\in\mathbb N$, we define
\[\widehat{h}^0(\mathscr X,\overline{\mathscr L}^{\otimes D})=\log\#\left\{s\in H^0(\mathscr X,\mathscr L^{\otimes D})\mid \|s\|_{\sup,v}\leqslant1,\;\forall\; v\in M_{K,\infty}\right\},\]
and
\[\widehat{\vol}(\overline{\mathscr L})=\limsup_{D\to\infty}\frac{\widehat{h}^0(\mathscr X,\overline{\mathscr L}^{\otimes D})}{D^{\dim(\mathscr X)}/\dim(\mathscr X)!}.\]
By the works mentioned above, the arithmetic volume function $\widehat{\vol}(\ndot)$ has many similar properties to those of the geometric setting $\vol(\ndot)$, so it is expected that an arithmetically big line bundle $\overline{\mathscr L}$, which means $\widehat{\vol}(\overline{\mathscr L})>0$, also has the property of Fujita approximation. In fact, this used to be a conjecture of A. Moriwaki.
\subsubsection{}
In \cite{Chen10,Yuan09}, H. Chen and X. Yuan proved the arithmetic Fujita approximation over number field at almost the same time. Although they both applied \cite[\S 3]{Lazarsfeld_Mustata08}, but their methods are essential different. X. Yuan \cite[Theorem C]{Yuan09} takes pluri-subharmonic envelope to deal with the infinite part, which is well-known in complex analysis.

In \cite{Chen10}, H. Chen applied his $\mathbb R$-filtration method developed in \cite{Chen10b}. More precisely, by considering the successive minima on $H^0(\mathscr X,\mathscr L^{\otimes n})$, we can make $\bigoplus\limits_{n\geqslant0}H^0(\mathscr X,\mathscr L^{\otimes n})$ to be a filtrated graded linear series. Then by combining some results of approximable graded linear series in \cite{Lazarsfeld_Mustata08}, we are able to apply some convergent properties of $\mathbb R$-filtrations to construct an ample line bundle which approximates $\overline{\mathscr L}$ well. This is the main idea of the proof in \cite{Chen10}.
\subsection{Arakelov geometry over adelic curves}
In \cite{ChenMoriwaki2020}, H. Chen and A. Moriwaki developed a kind of Arakelov theory over the so-called adelic curve. Based on this framework, they built an arithmetic intersection theory in \cite{ChenMoriwaki2021}, and studied the arithmetic positivities in \cite{ChenMoriwaki2022}. Under this setting, we can study the arithmetic geometry over number fields, function fields and other bases with adelic structures under a same framework.
\subsubsection{}
Briefly speaking, an adelic curve $S=(K,(\Omega,\mathcal A,\nu),\phi)$ consists of an arbitrary field $K$ which is equipped with a parameter set $\Omega$ of all its absolute values. We require the set $\Omega$ parameterized by $\phi$ has a $\sigma$-algebra structure $\mathcal A$ and a measure $\nu$ compatible with $\mathcal A$.

Under this setting, we are able to generalize many classic results and subjects of Arakelov geometry, including the arithmetic volume function, arithmetic positivity, and Hilbert-Samuel formula. Since the adelic structure becomes more complicated than those of classic cases, the obstructions original from measurability, integrability, and non-discreteness make the developing process more technique. In fact, some classic properties are not valid any longer in this framework, though many of these properties have been developed in \cite{ChenMoriwaki2020,ChenMoriwaki2021,ChenMoriwaki2022}, and see also \cite{Sedillot2023}.
\subsubsection{}
Inspired by \cite{Chen10,Yuan09}, it is quite natural to ask whether the Fujita approximation is still valid over adelic curves. Over a general adelic curve, we are lack of the Minkowskian property, so we cannot apply the filtrated method by successive minina directly in \cite{Chen10}. At the same time, some constructions and properties do not work any longer without the help of integral models or Minkowskian property, and see Example \ref{conter-exmaple of arithmetic Kodaira decompostion} for such an example. 

In \cite[\S 8.5]{ChenMoriwaki2022}, H. Chen and A. Moriwaki proved a similar result to the arithmetic Fujita approximation over adelic curves, which is called the \textit{relative Fujita approximation}. More precisely, let $\overline L$ be an adelic line bundle, where $L$ is big on $X$. Then for any $t<\wmu_{\max}^{\asy}(\overline L)$, there exists a sufficiently large $p\in\mathbb N$ and a birational morphism $\varphi:X'\rightarrow X$, such that $(\varphi^*\overline L)^{\otimes p}=\overline M\otimes \overline N$, where $\overline N$ is relatively ample and $\overline M$ is effective, and $\wmu_{\min}^{\asy}(\overline L)>pt$. Here $\wmu_{\max}^{\asy}(\ndot)$ and $\wmu_{\min}^{\asy}(\ndot)$ denote the asymptotic maximal and minimal slopes respectively.

Compared with this relative version above, the classic arithmetic Fujita approximation seems to take an "average" at every minima of minimal slope, and approximates all of them simultaneously, not only the maximal one. In fact, this is conjectured by Chen-Moriwaki in \cite[Remark 9.2.7]{ChenMoriwaki2022}.
\subsection{Main result}
In this article, we prove the following result on the arithmetic Fujita approximation over adelic curves under the framework \cite{ChenMoriwaki2020,ChenMoriwaki2022}, which gives a positive response to the conjecture \cite[Remark 9.2.7]{ChenMoriwaki2022}. More generally, we prove a version of adelic $\mathbb R$-line bundle below at the same time, which is also a generalization of \cite[Theorem 4.3]{Chen10} and \cite[Theorem C]{Yuan09}. 

In the statement of the main result, instead of considering an effective line bundle under the arithmetic setting, we use the notation that a line bundle has a non-zero small section over an adelic curve, and see \S \ref{definition of effective} for the precise definition.
\begin{theo}[Theorem \ref{arithmetic Fujita approximation of R-ample}]
Let $S=(K,(\Omega,\mathcal A,\nu),\phi)$ be a proper adelic curve with a perfect $K$, which satisfies either $\mathcal A$ is discrete or $K$ has a countable sub-field which is dense in $K_w$ for all $w\in\Omega$. Let $X$ be a projective variety over $\spec K$, and $\overline L$ be an arithmetically big $\mathbb R$-line bundle on $X$ over $S$. Then for any $\epsilon>0$, there exist a sufficiently large $p\in\mathbb N$ and a birational morphism $v:X' \rightarrow X$, such that $v^*(\overline L^{\otimes p})=\overline A\otimes \overline M$, where $\overline A$ is arithmetically ample and $\overline M$ is an adelic $\mathbb R$-line bundle have a non-zero small section. Moreover, we have
\[p^{-(\dim(X)+1)}\widehat{\vol}(\overline A)\geqslant\widehat{\vol}(\overline L)-\epsilon.\]
\end{theo}
\subsubsection{}
In \cite[\S 3]{ChenMoriwaki2022}, they develop a method which constructs a graded linear series over the adelic curve of trivially valued field from general one by its Harder-Narasimhan filtration. Next, we can modify the original norms by its spectral norm, which makes the graded linear series to be filtrated. Under this construction, the new graded linear series has some asymptotic invariants same as the original one. Then some ideas in \cite{Chen10,Lazarsfeld_Mustata08} can be absorbed in this work.

Once we prove the case of $\mathbb Q$-line bundle (in Corollary \ref{arithmetic Fujita approximation of Q-ample}) from the case of usual adelic line bundle (in Corollary \ref{arithmetic Fujita approximation of ample}), we can approximate an arithmetically big $\mathbb R$-line bundle by an arithmetically big $\mathbb Q$-line bundle in an effective direction, and then we can prove the case of $\mathbb R$-line bundle (in Theorem \ref{arithmetic Fujita approximation of R-ample}).

\subsubsection{}
In the application of $\mathbb R$-filtration method, we need to take some limits of some graded linear series and measures, and most operations only work when the objects are measurable and integrable. If either $\mathcal A$ is discrete or $K$ has a countable sub-field which is dense in $K_w$ for all $w\in\Omega$, there will be no obstructions on the measurability or integrability, and see \cite[\S 4]{ChenMoriwaki2020} for more details of the necessity.

In order to study some arithmetic positivities by $\wmu_{\max}^{\asy}(\ndot)$ and $\wmu_{\min}^{\asy}(\ndot)$, we need to suppose the base field $K$ is perfect. For example, for an adelic line bundle $\overline L$ where $L$ is nef, we need to apply $\wmu_{\min}^{\asy}(\overline L)\geqslant0$ to deduce $\overline L$ is arithmetically nef by \cite[Proposition 9.1.7]{ChenMoriwaki2022}, where we need the perfect condition. See \cite[\S 9]{ChenMoriwaki2022} for more details of the necessity.

Of course, in \cite{Chen10,Yuan09}, there is no such problem since for the case of number field, because $\mathcal A$ is discrete in the language of adelic curve and the base field is perfect by its characteristic $0$.
\subsection{Organization of article}
In \S \ref{chap2}, we introduce the required notions of Arakelov theory over adelic curves and $\mathbb R$-filtration. In \S \ref{Chap. 3}, we consider graded linear series over a general adelic curve and construct those over the adelic curve of trivially valued field filtrated by the Harder-Narasimhan filtration. In \S \ref{Chap4}, we prove the property of Fujita approximation of certain graded linear series. In \S \ref{Chap5}, we prove the arithmetic Fujita approximation of adelic line bundles in the usual case, the case of $\mathbb Q$-line bundle and the case of $\mathbb R$-line bundle sequentially.
\subsection*{Acknowledgements}
I would like to thank Prof. Atsushi Moriwaki for introducing this subject to me, and suggesting me to use the notation "having a non-zero small section". In addition, I would like to thank Prof. Huayi Chen for many useful discussions and suggestions, and Dr. Wenbin Luo for the advices on some key techniques. In particular, I would like to thank the anonymous referee for his or her valuable comments on this article. Part of this work was accomplished during my visit to Institute for Theoretical Sciences, Westlake University, and I would like to thank their generous invitation and reception. Chunhui Liu was supported by National Science Foundation of China No. 12201153.
\section{Arakelov theory over adelic curves}\label{chap2}
In this section, we will introduce the foundations of Arakelov theory over adelic curves and of $\mathbb R$-filtration. The main references are \cite{ChenMoriwaki2020,ChenMoriwaki2022}.
\subsection{Foundations of Arakelov geometry over adelic curves}
First, we recall some useful knowledge on the Arakelov geometry over adelic curves, and see \cite{ChenMoriwaki2020} for a systematical introduction to this subject.
\subsubsection{Adelic curve}
Let $K$ be a field. We denote by $M_K$ the set of all absolute values on $K$.
\begin{defi}[Adelic curve]
We say that $(\Omega,\mathcal A,\nu)$ is an \textit{adelic structure} on $K$ equipped with a mapping
\[\begin{array}{rrcl}
\phi:&\Omega&\longrightarrow&M_K\\
&\omega&\mapsto&|\ndot|_\omega
\end{array}\]
which satisfies the properties below:
\begin{enumerate}
\item[(1)]$\mathcal A$ is a $\sigma$-algebra on $\Omega$, and $\nu$ is a measure on $(\Omega,\mathcal A)$;
\item[(2)]for all $a\in K^\times$, the function
\[\begin{array}{rcl}
\Omega&\longrightarrow&\mathbb R\\
w&\mapsto&\log|a|_w
\end{array}\]
is $\mathcal A$-measurable and integrable with respect to $\nu$.
\end{enumerate}
In addition, the data $S=(K,(\Omega,\mathcal A,\nu),\phi)$ is called an \textit{adelic curve}.
\end{defi}
We denote by $\Omega_f$ the subset of $\Omega$ whose images in $M_K$ with respect to $\phi$ is non-archimedean, and by $\Omega_\infty$ the subset of $\Omega$ whose image in $M_K$ is archimedean. We have $\Omega=\Omega_f\sqcup\Omega_\infty$ by definition directly.
\begin{defi}[Proper adelic curve]
We say that an adelic curve $S=(K,(\Omega,\mathcal A,\nu),\phi)$ is \textit{proper} if for all $a\in K^\times$, we always have
\[\int_\Omega\log|a|_w\nu(\mathrm d w)=0,\]
and the above equality is called the \textit{product formula} of $S$.
\end{defi}
We refer the readers to \cite[\S 3]{ChenMoriwaki2020} for more details on these notions.
\begin{exem}
Let $K$ be a number field or a function field. Then the usual ad\`ele ring of $K$ gives an adelic structure, and constructs examples of proper adelic curves.

Let $K$ be a finitely transcendental extension over $\mathbb Q$, which is a sub-field of $\mathbb C$. In \cite{Moriwaki00}, A. Moriwaki constructed a special adelic structure over $K$, and defined a height function with respect to it. This height function has been applied to study many arithmetic geometry subjects.
\end{exem}
\begin{exem}[Trivially valued field]\label{example: trivially valued field}
Let $K$ be an arbitrary field, and $|\ndot|_0$ be the trivial absolute value over $K$, which means
\[|a|_0=\begin{cases}
        1,\text{ if }a\neq0,\\
        0,\text{ if }a=0.
        \end{cases}\]
We note $S_0=(K,\{0\},|\ndot|_0)$ the adelic curve which counts the absolute value $|\ndot|_0$ once. By definition, $S_0$ is a proper adelic curve.
\end{exem}
Example \ref{example: trivially valued field} will be used in the article later. We refer the readers to \cite[\S 3.2]{ChenMoriwaki2020} for more examples of adelic curves.
\subsubsection{Adelic vector bundle}
Let $S=(K,(\Omega,\mathcal A,\nu),\phi)$ be an adelic curve, and $E$ be a vector space over $K$. For a family of norms $(\|\ndot\|_w)_{w\in\Omega}$ on $E_w=E\otimes_{K}K_w$ over $K_w$, if the norm family $(\|\ndot\|_w)_{w\in\Omega}$ on $E$ and its dual norm family $(\|\ndot\|^\vee_w)_{w\in\Omega}$ on $E^\vee$ are both $\mathcal A$-measurable on $\Omega$, and $(\|\ndot\|_w)_{w\in\Omega}$ is dominated, then we say $\overline E=(E,(\|\ndot\|_w)_{w\in\Omega})$ is an \textit{adelic vector bundle} over $S$. If $\dim(E)=1$, we say that $\overline E$ is an \textit{adelic line bundle} over $S$. See \cite[Definition 4.1.28]{ChenMoriwaki2020} for more details on this notion.

We say that an adelic vector bundle $\overline F=(F,(\|\ndot\|'_w)_{w\in\Omega})$ is an \textit{adelic sub-bundle} of $\overline E=(E,(\|\ndot\|_w)_{w\in\Omega})$ over $S$, if $F$ is a $K$-vector subspace of $E$, and $\|\ndot\|'_w$ is the restriction of $\|\ndot\|_w$ on $F$ for each $w\in\Omega$.

An adelic bundle $\overline G=(G,(\|\ndot\|''_w)_{w\in\Omega})$ is said to be an \textit{adelic quotient bundle} of $\overline E=(E,(\|\ndot\|_w)_{w\in\Omega})$ over $S$, if $G$ is a quotient vector space of $E$, and $\|\ndot\|''_w$ is the quotient norm $\|\ndot\|_w$ with respect to the canonical quotient homomorphism $E\rightarrow G$ for each $w\in\Omega$. If $G$ is the quotient space of $E$ modulo $F$, and is equipped with the quotient metrics from $\overline F$, then we denote $\overline G=\overline E/\overline F=\overline{E/F}$.

For an adelic vector bundle $\overline E=(E,(\|\ndot\|_w)_{w\in\Omega})$ over $S$, we define $\det(\overline E)=(\det(E),(\|\ndot\|_{\det,w})_{w\in\Omega})$ as follows: the $K$-vector space $\det(E)=\bigwedge^{\dim(E)}(E)$ is the determinant space of $E$, and $\|\ndot\|_{\det,w}$ is the determinant norm on $\det(E)_w$ of $\|\ndot\|_w$ for each $w\in\Omega$. We refer the readers to \cite[\S 1.1.13, \S 1.2.3]{ChenMoriwaki2020} for the details of the construction of $\|\ndot\|_{\det,w}$.
\begin{rema}
In general, for an adelic vector bundle $\overline E$ over $S=(K,(\Omega,\mathcal A,\nu),\phi)$, we cannot make sure that all its subspace equipped with the subspace norms is still an adelic vector bundle, and its quotient spaces admit the same situation. However, if $\mathcal A$ is a discrete $\sigma$-algebra or $K$ has a countable sub-field which is dense in $K_w$ for all $w\in\Omega$, we will get adelic vector bundles from all the above constructions, and see \cite[Proposition 4.1.24]{ChenMoriwaki2020} for more details. In the whole article, we will always assume these conditions.
\end{rema}
\subsubsection{Some Arakelov invariants}
Let $S=(K,(\Omega,\mathcal A,\nu),\phi)$ be a proper adelic curve, and $\overline E=(E,(\|\ndot\|_w)_{w\in\Omega})$ be an adelic vector bundle over $S$. For the determinant bundle $\det(\overline E)=(\det(E),(\|\ndot\|_{\det,w})_{w\in\Omega})$ of $\overline E$, we select an $s\in\det(E)\smallsetminus\{0\}$, and define
\[\adeg(\overline E)=-\int_\Omega\log\|s\|_{\det,w}\nu(\mathrm d w),\]
which is well-defined due to \cite[\S 4.2]{ChenMoriwaki2020}. This is independent of the choice of $s$ by the product formula of $S$.

We define the \textit{positive Arakelov degree} of $\overline E$ as
\[\adeg_+(\overline E)=\sup_{F\subseteq E}\adeg(\overline F),\]
where $\overline F$ runs over all the adelic sub-bundles of $\overline E$.

For a non-zero adelic bundle $\overline E$, we define the \textit{slope} of $\overline E$ as
\[\wmu(\overline E)=\frac{\adeg(\overline E)}{\dim(E)}.\]

In addition, we define the \textit{maximal slope} of $\overline E$ as
\[\wmu_{\max}(\overline E)=\sup_{F\subseteq E}\wmu(\overline F),\]
where $\overline F$ runs over all adelic sub-bundles of $\overline E$. Similarly, we define the \textit{minimal slope} of $\overline E$ as
\[\wmu_{\min}(\overline E)=\inf_{E\twoheadrightarrow G}\wmu(\overline G),\]
where $\overline G$ runs over all adelic quotient bundles of $\overline E$.

We refer the readers to \cite[\S 4.3]{ChenMoriwaki2020} for more details and properties of these invariants.
\subsection{Harder-Narasimhan filtration from $\mathbb R$-filtration}
In this part, we introduce a method to filtrate an adelic vector bundle by its minimal slopes, which is called the Harder-Narasimhan filtration of an adelic vector bundle. We will use the language of $\mathbb R$-filtration to formulate this construction. See \cite{Chen10b} for its origin and \cite{ChenMoriwaki2020} for the version over adelic curve.
\subsubsection{$\mathbb R$-filtrations on vector spaces}
Let $K$ be a field, and $W$ be a $K$-vector space of finite dimension. For an \textit{$\mathbb R$-filtration} on $W$, we mean a sequence $\mathcal F=\left(\mathcal F^t W\right)_{t\in\mathbb R}$ of vector subspaces of $W$, which satisfies the following conditions:
\begin{enumerate}
\item[(1)] if $t\leqslant s$, then $\mathcal F^sW\subseteq \mathcal F^tW$;
\item[(2)] $\mathcal F^tW=0$ for sufficiently large $t\in\mathbb R$, and $\mathcal F^tW=W$ for sufficiently small $t\in\mathbb R$;
\item[(3)] the function
\[\begin{array}{rcl}
\mathbb R&\longrightarrow&\mathbb R\\
t&\mapsto&\dim_K(\mathcal F^tW)
\end{array}\]
is left continuous with respect to Euclidean topology.
\end{enumerate}
\subsubsection{A measure induced by $\mathbb R$-filtration}\label{measure from R-filtration}
Let $K$ be a field. For a $K$-vector space $W$ of finite dimension equipped with an $\mathbb R$-filtration $\mathcal F$, we denote by $\nu_{(W,\mathcal F)}$ (or by $\nu_W$ if there is no ambiguity) the Borel probability measure on $\mathbb R$, which is induced from the derivative of the function
\[\begin{array}{rcl}
\mathbb R&\longrightarrow&\mathbb R\\
t&\mapsto&\displaystyle-\frac{\dim(\mathcal F^tW)}{\dim(W)}.
\end{array}\]
In fact, $\nu_{(W,\mathcal F)}$ is a linear combination of some Dirac measures. More precisely, if the $\mathbb R$-filtration $\mathcal F$ consists of a flag
\[W=W_0\supsetneq W_1\supsetneq\cdots\supsetneq W_n=0\]
together with the jumps at the real numbers $\lambda_1<\cdots<\lambda_n$, then we have
\[\nu_{(W,\mathcal F)}=\sum_{i=1}^n\frac{\dim\left(W_{i-1}/W_i\right)}{\dim(W)}\delta_{\lambda_i},\]
where $\delta_{\lambda_i}$ is the Dirac measure at $\lambda_i\in\mathbb R$.

Next, we define the function
\begin{equation}\label{norm induced by filtration}
\begin{array}{rrcl}
\lambda:& W & \longrightarrow & \mathbb R\cup\{+\infty\} \\
 & x & \mapsto & \sup\{t\in\mathbb R\mid x\in\mathcal F^t W\}.
\end{array}\end{equation}
Then we have $\lambda(W)\subset\supp(\nu_W)\cup\{+\infty\}$, and $\lambda(x)$ is finite if $x\neq0$.

In addition, we define
\begin{equation}\label{lambda_+ and lambda_max}
\lambda_+(W)=\int_0^{+\infty}x\nu_W(\mathrm d x)\text{, and }\lambda_{\max}(W)=\max\limits_{x\in W\smallsetminus\{0\}}\lambda(x).
\end{equation}
\subsubsection{Harder-Narasimhan filtration}
By the notion of $\mathbb R$-filtration, we introduce a filtration by the minimal slopes of adelic bundles. This is original from an analogical study of vector bundles on a smooth projective curve \cite{Harder-Nara}.
\begin{defi}[Harder-Narasimhan filtration]\label{definition of HN-filtration}
Let $\overline E$ be a non-zero adelic vector bundle over a proper adelic curve $S=(K,(\Omega, \mathcal A,\nu),\phi)$, where either $\mathcal A$ is a discrete $\sigma$-algebra or $K$ has a countable sub-field which is dense in $K_w$ for all $w\in\Omega$. For any $t\in\mathbb R$, let
\begin{equation*}
\mathcal F_{\hn}^t\overline E=\bigcap_{\epsilon>0}\sum_{\begin{subarray}{c}\{0\}\neq F\in\Theta(E)\\ \wmu_{\min}(\overline F)\geqslant t-\epsilon\end{subarray}} F=\sum_{\begin{subarray}{c}\{0\}\neq F\in\Theta(E)\\ \wmu_{\min}(\overline F)\geqslant t\end{subarray}} F,
\end{equation*}
where $\Theta(E)$ denotes the set of all vector subspaces of $E$, and $\overline F$ is equipped with the subspace norms. This $\mathbb R$-filtration $\mathcal F_{\hn}$ is called the \textit{Harder-Narasimhan filtration} of $\overline E$.
\end{defi}
The last equality in Definition \ref{definition of HN-filtration} is from \cite[Proposition 4.3.38]{ChenMoriwaki2020}. In fact, the Harder-Narasimhan filtration of $\overline E$ corresponds to an increasing flag of vector subspaces
\begin{equation}\label{HN flag}
\{0\}=E_1\subsetneq E_2\subsetneq\cdots\subsetneq E_n=E
\end{equation}
and a decreasing sequence of real numbers $\mu_1>\cdots>\mu_n$ represents the jumps of the $\mathbb R$-filtration $(\mathcal F_{\hn}^t\overline E)_{t\in\mathbb R}$, where
\[\mu_i=\sup\left\{t\in\mathbb R\mid \dim_K(\mathcal F^t_{\hn}\overline E)\geqslant i\right\}.\]
In particular, still by \cite[Proposition 4.3.38]{ChenMoriwaki2020}, we have $\mu_n=\wmu_{\min}(\overline E)$. The uniqueness of the flag \eqref{HN flag} is from \cite[Theorem 4.3.58]{ChenMoriwaki2020}.
\subsection{Filtrations on graded algebras}\label{appromable graded algebra}
In this part, we study some asymptotic properties of graded linear algebras by $\mathbb R$-filtration.
\subsubsection{Approximable graded algebras}
The notation of approximable graded algebra is from \cite{Chen10}, whose idea is original from \cite{Lazarsfeld_Mustata08}.
\begin{defi}[Approximable graded algebra]\label{definition of approximable algebra}
Let $K$ be a field, and $E_\bullet=\bigoplus\limits_{D\geqslant0}E_D$ be a graded $K$-algebra. We define that $E_\bullet$ is \textit{approximable} if:
\begin{enumerate}
\item[(1)] for each $D\in\mathbb N$, $E_D$ is of finite dimension, and $E_D\neq0$ for sufficiently large $D\in\mathbb N$;
\item[(2)] for every $\epsilon\in(0,1)$, there exists a $p_0\in\mathbb N^+$, such that for any $p\geqslant p_0$, we have
\[\liminf\limits_{n\to\infty}\frac{\dim_K\left(\im\left(S^nE_p\rightarrow E_{np}\right)\right)}{\dim_K(E_{np})}>1-\epsilon,\]
where $S^nE_p\rightarrow E_{np}$ is the canonical homomorphism defined by the algebra structure on $E_{\bullet}$.
\end{enumerate}
\end{defi}

\subsubsection{Convergence of measures on an approximable algebra}\label{convergence of approximable algebras}
For a graded algebra, we equip an $\mathbb R$-filtration with its each homogeneous part. If it is an approximable graded algebras, we will have some convergent properties.

Let $E_\bullet=\bigoplus\limits_{D\geqslant0}E_D$ be an approximable graded algebra. Suppose for every $D\geqslant0$, the vector space $E_D$ is equipped with an $\mathbb R$-filtration $\mathcal F$. We ignore the index $D\in\mathbb N$ since there will be no ambiguity.
\begin{defi}\label{filtrated graded linear series}
With all the above notations and conditions. If for every homogeneous elements $x_1,x_2\in E_\bullet$, we have
\[\lambda(x_1 x_2)\geqslant\lambda(x_1)+\lambda(x_2),\]
where the function $\lambda(\ndot)$ is defined at \eqref{norm induced by filtration} and the homogeneous degrees of $x_1$ and $x_2$ are ignored, then we say $E_\bullet$ is \textit{filtrated} with respect to $\mathcal F$.
\end{defi}
For every $\epsilon>0$, let $T_\epsilon$ be the operator on the space of Borel probability measures on $\mathbb R$ by mapping $x\mapsto\epsilon x$.

We have the following result of filtrated graded algebras.
\begin{prop}[Theorem 2.11, \cite{Chen10}]\label{convergence of measures}
With all the above notations. Suppose that $E_\bullet=\bigoplus\limits_{D\geqslant0}E_D$ is a filtrated approximable graded algebra with respect to $\lambda(\ndot)$ defined at \eqref{norm induced by filtration}, and
\[\sup_{n\geqslant1}\frac{\lambda_{\max}(E_n)}{n}<+\infty,\]
where $\lambda_{\max}(\ndot)$ is defined at \eqref{lambda_+ and lambda_max}. Then the sequence $\left(\frac{\lambda_{\max}(E_n)}{n}\right)_{n\in\mathbb N^+}$ converges in $\mathbb R$, and the sequence of measures $\left(T_{\frac{1}{n}}(\nu_{E_n})\right)_{n\geqslant1}$ converges weakly to a Borel probability measure on $\mathbb R$.
\end{prop}
\subsubsection{Limits of a measure from filtrated algebras}
We consider the filtrated approximable graded $K$-algebra $E_\bullet=\bigoplus\limits_{D\geqslant0}E_D$. Let $\nu_{E_\bullet}$ be the limit of the sequence of measures $\left(T_{\frac{1}{n}}(\nu_{E_n})\right)_{n\geqslant1}$, and $\lambda^{\asy}_{\max}(E_\bullet)=\lim\limits_{n\to\infty}\frac{\lambda_{\max}(E_n)}{n}$, whose existences are assured by Proposition \ref{convergence of measures}. In addition, we have
\begin{equation}\label{asymptotic lambda+}
\lambda_+^{\asy}(E_\bullet):=\int_0^{+\infty}x\nu_{E_\bullet}(\mathrm dx)=\lim\limits_{n\to\infty}\frac{\lambda_+(E_n)}{n},
\end{equation}
where $\lambda_+(\ndot)$ is defined at \eqref{lambda_+ and lambda_max}. See \cite[\S 2.4]{Chen10} for more details.

We have the following result.
\begin{prop}[Theorem 2.13, \cite{Chen10}]\label{asymptotic maximal of lambda}
Let $E_\bullet$ be a filtrated approximable graded $K$-algebra. Then $\lambda_+^{\asy}(E_\bullet)>0$ if and only if $\lambda_{\max}^{\asy}(E_\bullet)>0$.
\end{prop}
\section{Graded linear series of global section spaces}\label{Chap. 3}
Let $S$ be an adelic curve, $X$ be a projective scheme over $\spec K$, and $\overline{L}$ be an adelic line bundle on $X$ over $S$. For our purpose, we will construct some graded linear series consisting of some particular adelic vector bundles over $S$ from above data.
\subsection{Adelic line bundles on an arithmetic variety}
First, we recall some notions about arithmetic positivities of the adelic line bundles on arithmetic variety over an adelic curve, which are original from \cite{ChenIkoma2020, ChenMoriwaki2022}. We refer the readers to \cite[\S 6 -- \S 9]{ChenMoriwaki2022} for a self-contained introduction to these notions.
\subsubsection{Adelic line bundle}\label{adelic line bundle on variety}
Let $S=(K,(\Omega,\mathcal A, \nu),\phi)$ be an adelic curve, and $X$ be a projective scheme over $\spec K$. We say that $\overline L=(L,\varphi)$ is an \textit{adelic line bundle} on $X$, if $L$ is a line bundle on $X$, and $\varphi=(\|\ndot\|_w)_{w\in\Omega}$ is a family of metrics on $L_w$ for each $w\in\Omega$, which is dominated and measurable.

In addition, we suppose that either $\mathcal A$ is a discrete $\sigma$-algebra or $K$ contains a countable sub-field which is dense in $K_w$ for all $w\in\Omega$. Then we construct a family of norms on $H^0(X,L)$ in the way below: for every $w\in\Omega$ and $x\in X_w^{\an}$, define
\[\begin{array}{rrcl}
|\ndot|_{\varphi_w}(x):&H^0(X_w,L_w)&\longrightarrow&\mathbb R\\
&s&\mapsto&|s|_{\varphi_w}(x)
\end{array}\]
as the value of the global section $s$ at the point $x$ with respect to the given metric, and
\begin{equation}\label{definition of sup norm}
\begin{array}{rrcl}
\|\ndot\|_{\varphi_w}:&H^0(X_w,L_w)&\longrightarrow&\mathbb R\\
&s&\mapsto&\sup\limits_{x\in X^{\an}_w}\{|s|_{\varphi_w}(x)\}.
\end{array}
\end{equation}
We denote $\varphi=(\|\ndot\|_{\varphi_w})_{w\in\Omega}$ the norm family on $H^0(X,L)$ over $S$ if there is no ambiguity. By \cite[Theorem 6.1.13, Theorem 6.1.32]{ChenMoriwaki2020}, $(H^0(X,L),\varphi)$ is an adelic vector bundle over $S$.

Let $\overline L=(L,\varphi)$ and $\overline M=(M,\psi)$ be two adelic line bundles on $X$. We denote by $\overline L\otimes \overline M=(L\otimes M,\varphi+\psi)$ the tensor adelic line bundle, where $\varphi+\psi$ is the tensor product of metrics. In particular, for $n\in\mathbb N$, we denote by $\overline L^{\otimes n}=(L^{\otimes n},n\varphi)$ the adelic line bundle of $n$-times tensor product of $\overline L$, and by $(H^0(X,L^{\otimes n}),n\varphi)$ the adelic vector bundle over $S$ defined above. By \cite[Proposition 6.1.28, Proposition 6.1.30]{ChenMoriwaki2020}, all the above constructions are all well-defined.

We refer the readers to \cite[\S 6.1]{ChenMoriwaki2020} for a precise explanation of \S \ref{adelic line bundle on variety}, see also \cite[\S 2.6 -- \S 2.8]{ChenMoriwaki2022}.
\subsubsection{Semi-positive metric}
Let $S=(K,(\Omega,\mathcal A, \nu),\phi)$ be an adelic curve, $X$ be a projective scheme over $\spec K$, $L$ be a line bundle on $X$, $\varphi=(\varphi_w)_{w\in\Omega}$ and $\psi=(\psi_w)_{w\in\Omega}$ be two families of metrics on $L$ such that $(L,\varphi)$ and $(L,\psi)$ are both adelic line bundles. Then we define the \textit{distance} between $\varphi$ and $\psi$ as
\[d(\varphi,\psi)=\int_\Omega \sup_{x\in X^{\an}_w}\left|\log \frac{|\ndot|_{\varphi_w}(x)}{|\ndot|_{\psi_w}(x)}\right|\nu(\mathrm d w).\]
If $L$ is semi-ample, and if there exist a positive integer $m$ and a sequence $(\varphi_n)_{n\in\mathbb N}$ of quotient metric families (where $\varphi_n$ a metric family of $L^{\otimes mn}$), such that
\[\lim_{n\to\infty}\frac{1}{mn}d(nm\varphi,\varphi_n)=0,\]
we say that the metric family $\varphi$ is \textit{semi-positive}. See \cite[Definition 2.2.21]{ChenMoriwaki2020} for the details of the construction of a quotient norm family.
\subsubsection{Ampleness}\label{definition of ample}
Let $\overline L=(L,\varphi)$ be an adelic line bundle on $X$ over $S$. We say $\overline L$ is \textit{relatively ample} if $L$ is ample and $\varphi$ is semi-positive.

In addition, an adelic line bundle $\overline L$ on $X$ is \textit{ample}, if it is relatively
ample and if there exists an $\epsilon>0$ such that the inequality
\[(\overline L|_Y^{\dim(Y)+1})\geqslant\epsilon\deg_L(Y)(\dim(Y)+1)\]
holds for any integral closed subscheme $Y$ of $X$, where $(\overline L|_Y^{\dim(Y)+1})$ is the arithmetic intersection number studied in \cite{ChenMoriwaki2021}. We also call it \textit{arithmetically ample} to distinguish it from the geometric case.
\begin{rema}
We compare the above definition with the arithmetically ample definition of S. Zhang in \cite{Zhang95}, where an arithmetically ample Hermitian line bundle on an arithmetic variety requires the metrics on the line bundle to be smooth. The definition in \cite[Definition 9.1.1]{ChenMoriwaki2022} over an adelic curve only require continuous metrics. In this article, we follow the definition given in \cite[Definition 9.1.1]{ChenMoriwaki2022} as above.
\end{rema}
\subsubsection{Nefness}\label{definition of nef}
Let $\overline L=(L,\varphi)$ be an adelic line bundle on $X$ over $S$. We say that $\overline L$ is \textit{relatively nef} if there exists a relatively ample adelic line bundle $\overline A$ on $X$ and a positive integer $N$ such that, for any $n\geqslant N$, the tensor product of adelic line bundles $\overline L^{\otimes n}\otimes \overline A$ is relatively ample.

An adelic line bundle $\overline L$ is said to be \textit{nef} or \textit{arithmetically nef}, if there exist an arithmetically ample line bundle $\overline A$ and a positive integer $M$ such that $\overline L^{\otimes m}\otimes \overline A$ is ample for any $m\geqslant M$ in the arithmetic sense.
\subsubsection{Effectivity}\label{definition of effective}
Let $\overline L=(L,\varphi)$ be an adelic line bundle on $X$ over $S$. If $s\in H^0(X,L)$ is a non-zero global section which satisfies $\|s\|_{\varphi_w}\leqslant 1$ for all $w\in\Omega$, we say that the global section $s$ is \textit{small}, or $s$ is a \textit{small section}.

We say that an adelic line bundle $\overline L$ \textit{has a non-zero small section}, if it admits at least one small global section by the above definition.
\begin{rema}
In \cite[Definition 8.5.4]{ChenMoriwaki2022}, the above $\overline L$ is said to be \textit{effective} in the arithmetic sense, but this notion is only used in \S 8.5 and \S 8.6 of \cite{ChenMoriwaki2022}. In order to avoid the confusion, we use the word “$\overline{L}$ has a non-zero small section” instead of \cite[Definition 8.5.4]{ChenMoriwaki2022} in this article, because some multiple of an arithmetically big adelic line bundle is not necessarily effective. This is different from the classic case over a number field. 
\end{rema}
\subsection{Algebraic construction of graded linear series}
In this part we will only consider the geometric setting, which means everything works over an arbitrary field with no metric. For every $D\in\mathbb N^+$, we consider the global section space $H^0(X,L^{\otimes D})$, and the graded $K$-algebra $\bigoplus\limits_{n\geqslant 0}H^0(X,L^{\otimes n})$.
\begin{defi}[Graded linear series]
Let $X$ be a projective variety over $K$, and $L$ be a line bundle on $X$. For a \textit{graded linear series} of $L$, we mean a graded sub-$K$-algebra of $\bigoplus\limits_{n\geqslant 0}H^0(X,L^{\otimes n})$.
\end{defi}
\subsubsection{Graded linear series containing an ample divisor}
The following definition is original from \cite[Definition 2.9]{Lazarsfeld_Mustata08}, where it is called "Condition (C)". This notion plays a key role in the arithmetic Fujita approximation of \cite{Chen10}, and we follow the notation of \cite[Definition 3.1]{Chen10}.
\begin{defi}\label{contain an ample divisor}
Let $X$ be a projective variety, and $L$ be a line bundle on $X$. Suppose that there exists a $p\in\mathbb N^+$, an ample line bundle $A$, an effective line bundle $M$ on $X$, and a non-zero section $s\in H^0(X,M)$, such that $L^{\otimes p}=A\otimes M$. In this case, we say a graded linear series $W_\bullet=\bigoplus\limits_{n\geqslant0}W_n$ of $L$ \textit{contains an ample divisor} if the homomorphism of graded algebras
\[\bigoplus_{n\geqslant0}H^0(X,A^{\otimes n})\rightarrow \bigoplus_{n\geqslant0}H^0(X,L^{\otimes np})\]
induced by $s$ can factor through $\bigoplus\limits_{n\geqslant0}W_{np}$.
\end{defi}
In \S \ref{convergence of approximable algebras}, we have seen that a filtrated approximable graded linear series has many convergent properties. The following property is original from \cite{Lazarsfeld_Mustata08}, which is applied in \cite{Chen10} to judge whether a graded linear series is approximable.
\begin{prop}[Theorem 3.4, \cite{Chen10}]\label{contain an ample divisor => approximable}
If $W_\bullet=\bigoplus\limits_{n\geqslant0}W_n$ contains an ample divisor and $W_n\neq0$ for sufficiently large $n$, then $W_\bullet$ is approximable.
\end{prop}
\subsubsection{Big line bundles}\label{geometric big graded linear series}
Let $W_\bullet=\bigoplus\limits_{n\geqslant0}W_n$ be a graded linear series of $L$. We denote
\[\vol(W_\bullet)=\limsup_{n\to\infty}\frac{\dim_K(W_n)}{n^{\dim(X)}/\dim(X)!}.\]
If $W_n=H^0(X,L^{\otimes n})$, we denote $\vol(L)=\vol(W_\bullet)$.

If $\vol(L)>0$, we say that $L$ is a \textit{big} line bundle on $X$.
\subsection{Arithmetic invariants of graded linear series}
In this part, we consider some arithmetic constructions.
\subsubsection{}
Let $\overline L=(L,\varphi)$ be an adelic line bundle on $X$ over a proper adelic curve $S=(K,(\Omega,\mathcal A,\nu),\phi)$, where either $\mathcal A$ is a discrete $\sigma$-algebra or $K$ contains a countable sub-field which is dense in $K_w$ for all $w\in\Omega$. In \S \ref{adelic line bundle on variety}, we defined a norm family on $H^0(X,L^{\otimes D})$ by \eqref{definition of sup norm} indexed by $\Omega$, which is denoted by $D\varphi$ with no ambiguity. By \cite[\S 6.4.1]{ChenMoriwaki2020}, $\|\ndot\|_{D\varphi_w}$ is ultrametric for every $w\in\Omega_f$. Then $(H^0(X,L^{\otimes D}),D\varphi)$ is an adelic vector bundle over $S$.

For every $D\in\mathbb N$, let $\overline E_D$ be an adelic vector sub-bundle of $H^0(X,L^{\otimes D})$ over $S$. In this case, we say that $\overline E_{\bullet}=\bigoplus\limits_{D\in\mathbb N}\overline E_D$ is a \textit{normed graded linear series} over $S$ if $E_D\cdot E_{D'}\subseteq E_{D+D'}$ for $D,D'\in\mathbb N$.
\subsubsection{}
Let $\overline E_{\bullet}=\bigoplus\limits_{D\in\mathbb N}\overline E_D$ be a normed graded linear series over $S$. We define the arithmetic volume of $\overline E_\bullet$ as
\begin{equation}
\widehat{\vol}(\overline E_\bullet)=\limsup_{n\to\infty}\frac{\adeg_+\left(\overline E_n\right)}{n^{\dim(X)+1}/(\dim(X)+1)!}.
\end{equation}
In particular, if $\overline E_n=\left(H^0(X,L^{\otimes n}),n\varphi\right)$ as above, we denote $\widehat{\vol}(\overline L)=\widehat{\vol}(\overline E_\bullet)$ for simplicity.

If $\widehat{\vol}(\overline L)>0$, we say that $\overline L$ is \textit{big} on $X$, or \textit{arithmetically big} to distinguish it from the geometric case.

\subsubsection{}
For every $D\in\mathbb N$, let $\left(\mathcal F^t_{\hn}(\overline E_D)\right)_{t\in\mathbb R}$ be the $\mathbb R$-filtration induced by the Harder-Narasimhan filtration of $\overline E_D$ following Definition \ref{definition of HN-filtration}, where we ignore the homogeneous degree $D\in\mathbb N$ of $\overline E_D$.

We denote by $\nu_{\overline E_D}$ the Borel probability measure induced by the Harder-Narasimhan filtration of $\overline E_D$ constructed in the sense of \S \ref{measure from R-filtration}. Then we have the result below from Proposition 4.3.50, Proposition 4.3.51 and Corollary 4.3.52 of \cite{ChenMoriwaki2020}, see also \cite[\S 2.5.4]{ChenMoriwaki2022}.
\begin{prop}\label{Estimates by R-filtration}
With all the above notations and conditions. We have
\[0\leqslant\adeg_+(\overline E_D)-\int_0^{+\infty}x\nu_{\overline E_D}(\mathrm dx)\leqslant\frac{1}{2}\nu(\Omega_\infty)\dim_K(E_D)\log(\dim_K(E_D)).\]
\end{prop}

\subsection{Graded linear series over trivially valued fields}
Let $K$ be a field, $S=(K,(\Omega,\mathcal A,\nu),\phi)$ be a proper adelic curve where either $\mathcal A$ is a discrete $\sigma$-algebra or $K$ has a sub-field which is dense in $K_w$ for all $w\in\Omega$, and $S_0=(K,\{0\},|\ndot|_0)$ be the adelic curve of trivially valued field introduced in Example \ref{example: trivially valued field}.

In fact, for any normed vector space over $(K,|\ndot|_0)$ of finite dimension, its norm takes only finitely many values (cf. \cite[\S 3.1]{ChenMoriwaki2022}).

In this part, we will study the adelic vector bundles and graded linear series over $S_0$ following that in \cite[\S 3]{ChenMoriwaki2022}. Some constructions are from those over $S$.
\subsubsection{}
Let $\overline E$ be an adelic vector bundle over an adelic curve $S=(K,(\Omega,\mathcal A,\nu),\phi)$, and $\left(\mathcal F^t_{\hn}(\overline E)\right)_{t\in\mathbb R}$ be the Harder-Narasimhan filtration of $\overline E$ defined in Definition \ref{definition of HN-filtration}.
 We define the norm $\|\ndot\|_{\hn}$ on $E$ over $(K,|\ndot|_0)$ as
\begin{equation}\label{Harder-Narasimhan norm over S_0}
\begin{array}{rrcl}
\|\ndot\|_{\hn}:&E&\longrightarrow&\mathbb R^+\\
&s&\mapsto&\exp\left(-\sup\{t\in\mathbb R\mid s\in\mathcal F^t_{\hn}(\overline E)\}\right).
\end{array}
\end{equation}
Then $\overline E_{\hn}=\left(E,\|\ndot\|_{\hn}\right)$ is an adelic vector bundle over $S_0$.

\subsubsection{}\label{from original series to spectral norm series}
Let $\overline E_{\bullet}=\bigoplus\limits_{D\in\mathbb N}\overline E_D$ be a normed graded linear series over an adelic curve $S=(K,(\Omega,\mathcal A,\nu),\phi)$. For every $\overline E_D$, we denote by $\|\ndot\|_{D}^{\hn}$ the norm \eqref{Harder-Narasimhan norm over S_0} on $E_D$ indexed by $D\in\mathbb N$. Then $(E_D,\|\ndot\|_D^{\hn})$ an adelic vector bundle over $S_0$.

By \cite[Proposition 4.1.3]{ChenMoriwaki2022}, for every $s\in E_m$ and $t\in E_n$, we have
\[\|s\cdot t\|^{\hn}_{m+n}\leqslant\exp\left(f(\dim E_m)+f(\dim E_n)\right)\|s\|^{\hn}_{m}\|t\|^{\hn}_{n},\]
where $f(n)=\frac{3}{2}\nu(\Omega_\infty)\log n$, and $\nu(\Omega_\infty)$ is the measure of archimedean absolute values of $S=(K,(\Omega,\mathcal A,\nu),\phi)$. The term $f(n)$ is original from the partial response of the semi-stable conjecture of Bost \cite[Theorem A]{Bost-Chen}, see also \cite[\S 5.6]{ChenIkoma2020}.

In order to obtain a new family of norms which are sub-multiplicative, we introduce the \textit{spectral norm} $\|\ndot\|_{\sp,D}$ of $\|\ndot\|^{\hn}_D$ as \begin{equation}\label{definition of spectral norm}
        \begin{array}{rrcl}
        \|\ndot\|_{\sp,D}:&H^0(X,L^{\otimes D})&\longrightarrow&\mathbb R_{\geqslant 0}\\
        &s&\mapsto&\lim\limits_{n\rightarrow\infty}\left(\|s^n\|_{nD}^{\hn}\right)^{1/n}.
        \end{array}
        \end{equation}
In this case, $\left(E_D,\|\ndot\|_{\sp,D}\right)$ is an adelic vector bundle over $S_0$. By \cite[Proposition 3.2.2 (3)]{ChenMoriwaki2022}, for every $s\in E_m$ and $t\in E_n$, we have
\[\|s\cdot t\|_{\sp,m+n}\leqslant\|s\|_{\sp,m}\|t\|_{\sp,n}.\]

\subsubsection{}\label{normed graded linear series over S_0}
Let $\overline E_{\hn,D}=\left(E_D, \|\ndot\|_{\sp,D}\right)$, then
\begin{equation*}
        \overline E_{\hn,\bullet}=\bigoplus\limits_{D\in\mathbb N}\overline{E}_{\hn,D}
\end{equation*}
is a filtrated graded linear series (see Definition \ref{filtrated graded linear series} for the definition) over $S_0$, where we take $\lambda(\ndot)=-\log\|\ndot\|_{\sp,D}$ in Definition \ref{filtrated graded linear series}.

We denote by $\mathcal F_{\sp}$ the $\mathbb R$-filtration on $\overline{E}_{\hn,D}$ induced by the filtrated structure $\lambda(\ndot)=-\log\|\ndot\|_{\sp,D}$, where we omit the index $D\in\mathbb N$ in the above notation. The existence of such $\mathcal F_{\sp}$ is ensured by the finiteness of the values of $\|\ndot\|_{\sp,D}$ on $\overline E_{\hn,D}$.
\subsubsection{}
Some asymptotic behaviors of $\overline E_{\hn,\bullet}$ are useful for the study of $\overline E_\bullet$. Since $\overline E_{\hn,\bullet}$ is filtrated, Proposition \ref{convergence of measures} and Proposition \ref{asymptotic maximal of lambda} will be applied in this process.
\begin{prop}\label{same volume over S and S_0}
With all the notations in \S \ref{from original series to spectral norm series} and \S \ref{normed graded linear series over S_0}, we have
\[\widehat{\vol}(\E_\bullet)=\widehat{\vol}(\overline E_{\hn,\bullet}).\]
\end{prop}
\begin{proof}
It is from \cite[Corollary 3.4.4, (4.1.2), Corollary-Definition 4.1.4]{ChenMoriwaki2022} directly.
\end{proof}

By Proposition \ref{convergence of measures}, the sequence $\left(\frac{1}{n}\wmu_{\max}(\overline{E}_n)\right)_{n\geqslant 0}$ converges to a real number, and we denote it by $\wmu_{\max}^{\asy}(\overline{E}_\bullet)$. If $E_n=H^0(X,L^{\otimes n})$ equipped with the supremum norms, it is denoted by $\wmu_{\max}^{\asy}(\overline{L})$ for simplicity.

The following result is an equivalent property of arithmetically bigness.
\begin{prop}\label{criterion of arithmetically big}
Let $\overline L$ be an adelic line bundle on $X$ over a proper adelic curve $S$. If $L$ is a big line bundle on $X$, then $\widehat{\vol}(\overline L)>0$ if and only if $\wmu_{\max}^{\asy}(\overline{L})>0$. In addition, we have
\[\wmu_{\max}^{\asy}(\overline L)=\lambda_{\max}^{\asy}(\overline E_{\hn,\bullet}).\]
\end{prop}
\begin{proof}
The equivalence of $\widehat{\vol}(\overline L)>0$ and $\wmu_{\max}^{\asy}(\overline{L})>0$ is from \cite[Proposition 6.4.18]{ChenMoriwaki2020} directly. The equality is from \cite[Proposition 3.4.1]{ChenMoriwaki2022}.
\end{proof}
From Proposition \ref{Estimates by R-filtration} and Proposition \ref{same volume over S and S_0}, we have the following result by an elementary calculation.
\begin{prop}\label{computation of volume from asymptotic lambda+}
Let $\overline E_{\bullet}=\bigoplus\limits_{D\geqslant0}\overline E_D$ be a normed graded linear series over $S$. If $\vol(E_\bullet)>0$ following the definition in \S \ref{geometric big graded linear series}, then we have
\[\lambda_+^{\asy}(\overline E_{\hn,\bullet})=\frac{\widehat{\vol}(\overline E_\bullet)}{\left(\dim(X)+1\right)\vol(E_\bullet)},\]
where $\lambda_+^{\asy}(\ndot)$ is defined at \eqref{asymptotic lambda+} with respect to the canonical filtrated structure introduced in \S \ref{normed graded linear series over S_0} over $S_0$.
\end{prop}
\subsubsection{}
Let $\overline L=(L,\varphi)$ be an adelic line bundle on $X$ over an arbitrary proper curve $S$. We consider $\wmu_{\min}(H^0(X,L^{\otimes n}),n\varphi)$, and we define
\[\wmu_{\min}^{\asy}(\overline L)=\liminf_{n\to\infty}\frac{\wmu_{\min}(H^0(X,L^{\otimes n}),n\varphi)}{n},\]
where the adelic vector bundle $(H^0(X,L^{\otimes n}),n\varphi)$ over $S$ is defined in \S \ref{adelic line bundle on variety}.

We recall some results on the criteria of arithmetic positivities via $\wmu_{\max}^{\asy}(\ndot)$ and $\wmu_{\min}^{\asy}(\ndot)$ in \cite[\S 9]{ChenMoriwaki2022} below, where we always suppose that $K$ is perfect.
\begin{prop}[Proposition 9.1.2, \cite{ChenMoriwaki2022}]\label{properties of ample}
Let $\overline L$ be a relatively ample adelic line bundle. Then the following properties are equivalent:
\begin{enumerate}
\item[(1)] $\overline L$ is ample;
\item[(2)] $\wmu_{\min}^{\asy}(\overline L)>0$;
\item[(3)] there exists an $\epsilon>0$, such that for any closed sub-scheme $Y$ of $X$, we have $\wmu_{\max}^{\asy}(\overline L|_Y)>\epsilon$.
\end{enumerate}
See \S \ref{definition of ample} for the definitions of relatively ample and ample adelic line bundles.
\end{prop}
\begin{prop}[Proposition 9.1.7, \cite{ChenMoriwaki2022}]\label{properties of nef}
Let $\overline L$ be an adelic line bundle. Then the following properties are equivalent:
\begin{enumerate}
\item[(1)] $\overline L$ is nef;
\item[(2)] $\overline L$ is relatively nef and $\wmu_{\min}^{\asy}(\overline L)\geqslant 0$.
\end{enumerate}
See \S \ref{definition of nef} for the definitions of relatively nef and nef adelic line bundles.
\end{prop}
\section{Fujita approximation of graded linear series}\label{Chap4}
In this section, we will always denote $E_D=H^0(X,L^{\otimes D})$ for simplicity, and all the constructions will be based on this notation.

\subsection{Approximatible graded linear series}
Let $S=(K,(\Omega,\mathcal A,\nu),\phi)$ be a proper adelic curve, where either $\mathcal A$ is a discrete $\sigma$-algebra or $K$ has a countable which is dense in $K_w$ for all $w\in\Omega$. Let $S_0=(K,\{0\},|\ndot|_0)$ be the adelic curve of trivially valued field defined in Example \ref{example: trivially valued field}. In this part, we will consider some approximated properties of normed graded linear series over $S_0$ constructed from that over $S$.
\subsubsection{}
Let $\lambda\in\mathbb R_{\geqslant0}$, $\overline E_{\hn, 0}^{\lambda}=K$, and
\begin{equation}\label{definition of E_{hn,D}}
\overline E_{\hn, D}^{\lambda}=\Vect_K\left\{s\in H^0(X,L^{\otimes D})\mid \|s\|_{\sp,D}\leqslant\exp(-D\lambda)\right\}
\end{equation}
for $D\geqslant1$. So we have
\begin{equation*}
\overline E_{\hn,D}^\lambda=\mathcal F_{\sp}^{D\lambda}\overline E_{\hn,D},
\end{equation*}
where the $\mathbb R$-filtration $\mathcal F_{\sp}$ is defined in \S \ref{convergence of approximable algebras}, and $\overline E_{\hn, D}$ follows the definition in \S \ref{normed graded linear series over S_0} by taking $E_D=H^0(X,L^{\otimes D})$. In addition, we define
\begin{equation}\label{graded linear series modified by lambda}
\overline E^{\lambda}_{\hn,\bullet}=\bigoplus_{D\geqslant0}\overline E_{\hn,D}^\lambda,
\end{equation}
which is a graded linear series over $S_0$.

By the properties of $\overline E_{\hn,\bullet}$ in \S \ref{normed graded linear series over S_0}, we have the following proposition.
\begin{prop}\label{filtrated E_hn lambda}
The graded linear series \eqref{graded linear series modified by lambda} is a filtrated graded linear series over $S_0$.
\end{prop}
\subsubsection{}
Next, we prove that $\overline E_{\hn,D}^\lambda$ is approximable when $\lambda$ is of some particular values. For this target, we will prove it contains an ample divisor, and see Definition \ref{contain an ample divisor}.
\begin{prop}\label{contain an ample divisor}
Let $S=(K,(\Omega,\mathcal A,\nu),\phi)$ be a proper adelic curve, where either $\mathcal A$ is a discrete $\sigma$-algebra or $K$ has a countable sub-field which is dense in $K_w$ for all $w\in\Omega$. Let $\pi:X\rightarrow\spec K$ be a projective morphism, and $\overline L$ be an adelic line bundle over $S$. With all the notations and assumption same as in \eqref{graded linear series modified by lambda}. If $\overline L$ is big, then for all $\lambda\in\left(-\infty,\wmu^{\asy}_{\max}(\overline L)\right)$, $\overline E^{\lambda}_{\hn,\bullet}$ contains an ample divisor.
\end{prop}
\begin{proof}
By Proposition \ref{criterion of arithmetically big} and Proposition \ref{same volume over S and S_0}, from the fact that $\overline L$ is big, we have $\overline E_{\hn,\bullet}$ is big.

In addition, by the choice of $\lambda$, we have $\lambda^{\asy}_{\max}(\overline E_{\hn,\bullet}^\lambda)>0$ by definition directly. Then for a sufficiently large $n\in\mathbb N$, $L^{\otimes n}$ has a global section $s\in H^0(X,L^{\otimes n})$ satisfying $\|s\|_{\sp,n}\leqslant\exp(-n\lambda)$, which proves $\overline E_{\hn,n}^{\lambda}\neq0$ for $n\in\mathbb N$ sufficiently large.

By Kodaira decomposition \cite[Corollary 2.2.7]{LazarsfeldI}, there exists a $p_0\in\mathbb N$, such that for every $p\in\mathbb N\cap[p_0,+\infty)$, there always exist an ample line bundle $M$ and an effective line bundle $N$, which satisfies $L^{\otimes p}=M\otimes N$.

We may choose the above $p$ large enough, such that $\bigoplus\limits_{D\geqslant0}H^0(X,M^{\otimes D})$ is generated by global sections $t_1,\ldots,t_l\in H^0(X,M)$, which is also considered as sections in $H^0(X,L^{\otimes d})$.

We denote by $\lambda_0=\wmu_{\max}^{\asy}(\overline L)$ for simplicity. Then for all $\epsilon>0$, there exist a $k\in\mathbb N$ and an $s\in H^0(X,L^{\otimes k})$, such that
\[\|s\|_{\sp,pk}<\exp\left(-(\lambda_0-\epsilon/2)k\right).\]
Then there exists a $d\in\mathbb N$, such that
\[\|t_is^d\|_{\sp,kd+p}<\exp\left(-(\lambda_0-\epsilon)(kd+p)\right),\]
where $i=1,\ldots,l$.

Let $G_{\bullet}^{\epsilon}$ be the graded linear series generated by $s^Dt_1,\ldots,s^Dt_l$, which is a sub-graded algebra of $E_{\hn,\bullet}$. Then for $D\in\mathbb N$, we have
\[G_{D}^{\epsilon}=\Vect_K\left\{(s^dt_1)^{a_1}\cdots(s^dt_l)^{a_l}\mid a_1+\cdots+a_l=D\right\}=s^{dD}H^0(X,M^{\otimes D}),\]
which is a sub-$K$-algebra of $E_{\hn,(kd+p)D}$. Then for all $\lambda\in(-\infty,\wmu_{\max}^{\asy}(\overline L))$, we embed $H^0(X,M^{\otimes D})$ into $E_{\hn,pD}^{\lambda}$ via $s^{dD}$. Then the homomorphism
\[\bigoplus_{D\geqslant0}H^0(X,M^{\otimes D})\longrightarrow\bigoplus_{D\geqslant0}H^0(X,L^{\otimes pD})\]
factors through $\bigoplus\limits_{n\geqslant0} E_{\hn,np}^{\lambda}$.
\end{proof}

Same as the argument in \cite[Corollary 3.13]{Chen10}, we have the following result.
\begin{coro}\label{approximable of E_D^lambda}
We keep all the notations and conditions in Proposition \ref{contain an ample divisor}. The graded linear series \eqref{graded linear series modified by lambda} over $S_0=(K,\{0\},|\ndot|_0)$ is approximable for any $\lambda\in\left(-\infty,\wmu^{\asy}_{\max}(\overline L)\right)$.
\end{coro}
\begin{proof}
By Proposition \ref{contain an ample divisor => approximable}, we have the assertion from Proposition \ref{contain an ample divisor} directly.
\end{proof}
\subsubsection{}
In \cite[Proposition 3.11]{Chen10}, a similar result of Proposition \ref{contain an ample divisor} is proved by an arithmetic analogue of Kodaira decomposition \cite[Corollary 2.4]{Yuan2008} of X. Yuan. But it is not valid over a general adelic curve, and see Example \ref{conter-exmaple of arithmetic Kodaira decompostion} below for a counterexample. To avoid this obstruction, we use a method quite similar to that of \cite{Boucksom_Chen} to prove Proposition \ref{contain an ample divisor}, and see also \cite[\S 6.3]{ChenMoriwaki2020}.
\begin{exem}\label{conter-exmaple of arithmetic Kodaira decompostion}
Let $K$ be a field, and $S_3=(K,\{1,2,3\},|\ndot|_0)$, where the trivially valued absolute $|\ndot|_0$ (see Example \ref{example: trivially valued field} for the precise definition) is counted three times in a discrete $\sigma$-algebra. We consider the universal line bundle $\O(1)$ on $\mathbb P_K^1$, and we denote by $x_0,x_1$ the basis of $H^0(\mathbb P^1_K, \O(1))$ with respect to the homogeneous coordinate.

Let $\overline{\O(1)}=(\O(1),(\|\ndot\|_i)_{i=1}^3)$ be the adelic line bundle on $\mathbb P^1_K$, where the induced supremum norms $\|\ndot\|_{i=1}^3$ on $H^0(\mathbb P^1_K, \O(1))$ satisfies
\[\begin{cases}
\|x_i\|_1=e^{-\frac{1}{2},},\text{ where }i=0,1;\\
\|x_i\|_2=e^{-\frac{1}{2},},\text{ where }i=0,1;\\
\|x_i\|_3=e^{\frac{3}{4},},\text{ where }i=0,1.
\end{cases}\]
Then $\overline{\O(1)}$ is an arithmetically big line bundle on $\mathbb P^1_K$ over $S_3$, since \[\adeg\left(H^0(\mathbb P^1_K,\O(D)),(\|\ndot\|_i)_{i=1}^3\right)=\frac{1}{4}D(D+1)\]for all $D\in\mathbb N^+$.

If there exist a $D\in\mathbb N$, an adelic ample line bundle $\overline A$ and an line bundle $\overline M$ have a non-zero small section on $\mathbb P^1_K$, such that $\overline{\O(D)}=\overline A\otimes \overline M$, where $\overline{\O(D)}$ is equipped with the induced metrics from the one introduced above. Since $\Pic(\mathbb P^1_K)=\mathbb Z$, then there exists a $D'\in\mathbb N$ satisfying $D'<D+$, such that
\[\overline{\O(D)}=\overline{\O(D')}\otimes \overline{\O(D-D')},\]
where $\overline{\O(D')}$ is ample and $\overline{\O(D-D')}$ has a non-zero small section. By the definition of arithmetically ample line bundle, all the global sections of $\overline{\O(D')}$ should be of the norms smaller than $1$. But by the choice of $\|\ndot\|_3$ in $H^0(\mathbb P^1_K, \O(1))$ over $|\ndot|_0$, all the global sections of $\overline{\O(D-D')}$ should be equipped with a norm strictly larger than $1$ with respect to the third $\|\ndot\|_0$ in $S_3$. So there will be no small section in $\overline{\O(D-D')}$, which shows there is no arithmetic Kodaira decomposition in this case.
\end{exem}
\subsection{Fujita approximation of graded linear series}
Let $p\in\mathbb N$ such that $\overline{E}_{\hn,p}^0\neq0$, where we follow the notation in \eqref{definition of E_{hn,D}}. In this case, we define $\overline E_\bullet^{(p)}=\bigoplus\limits_{n\geqslant 0}\overline E^{(p)}_{\hn,np}$ to be the graded $K$-algebra generated by $\overline{E}_{\hn,p}^0$, which is considered as a normed graded linear series over $S_0=(K,\{0\},|\ndot|_0)$ defined in Example \ref{example: trivially valued field}.

The following result is an analogue of \cite[Theorem 4.1]{Chen10} over $S_0$, whose proof is quite similar to the ancient one.
\begin{prop}\label{Fujita approximation of a graded linear series}
Let $X$ be a projective variety over $\spec K$, and $\overline L$ be an arithmetically big line bundle on $X$ over a proper adelic curve $S=(K,(\Omega,\mathcal A,\nu),\phi)$, where either $\mathcal A$ is a discrete $\sigma$-algebra or $K$ has a countable sub-field which is dense in $K_w$ for all $w\in\Omega$. With all the above notations and constructions, we have
\[\widehat{\vol}(\overline L)=\sup_{p\in\mathbb N}\widehat{\vol}(\overline E_\bullet^{(p)}).\]
\end{prop}
\begin{proof}
For any $n\in\mathbb N^+$, let $\nu_{\sp,n}=T_{\frac{1}{n}}\nu_{\left(E_{\hn,n},\mathcal F_{\sp}\right)}$, where the notation $T_{\frac{1}{n}}\nu_{\left(E_{\hn,n},\mathcal F_{\sp}\right)}$ is introduced in \S \ref{convergence of approximable algebras}, and the $\mathbb R$-filtration $\mathcal F_{\sp}$ is defined in \S \ref{from original series to spectral norm series} with respect to the spectral norms. Then by Proposition \ref{convergence of measures} and Corollary \ref{approximable of E_D^lambda}, the sequence $(\nu_{\sp,n})_{n\geqslant 1}$ converges weakly to a Borel probability measure, which is denoted by $\nu_{\sp}$.

Similarly, let $\nu_{\sp,n}^{(p)}=T_{\frac{1}{np}}\nu_{\left(E^{(p)}_{\hn,np},\mathcal F_{\sp}\right)}$, which also converges weakly to a Borel probability measure noted by $\nu_{\sp}^{(p)}$ by Proposition \ref{filtrated E_hn lambda}.

In order to prove the assertion, we will prove that the restriction of $\nu_{\sp}$ on $[0,+\infty)$ can be approximated well by $\nu_{\sp}^{(p)}$ for $p\in\mathbb N$.

\textbf{Step I. -- Simultaneous approximations.}

Let
\[D:\;0=t_0<t_1<\cdots<t_m<\wmu_{\max}^{\asy}(\overline E_{\hn,\bullet})\]
be a division of the interval $[0,\wmu_{\max}^{\asy}(\overline E_{\hn,\bullet}))$ such that
\begin{equation}\label{zero measure condition of division}
\nu_{\sp}(\{t_1,\ldots,t_m\})=0.
\end{equation}
We define the function
\[\begin{array}{rrcl}
h_D:&\mathbb R&\longrightarrow&\mathbb R\\
&x&\mapsto&\sum\limits_{i=0}^{m-1}t_i\mathbbm{1}_{[t_i,t_{i+1})}(x)+t_m\mathbbm{1}_{[t_m,+\infty)}(x).
\end{array}\]
By Corollary \ref{approximable of E_D^lambda} and Definition \ref{definition of approximable algebra}, for every $\epsilon>0$, there exists a $p\in\mathbb N$ sufficiently large, such that $E_\bullet^{(p)}$ approximates all $E^{t_i}_{\hn,\bullet}$ defined at \eqref{graded linear series modified by lambda} simultaneously, where $i=0,\ldots,m$. In other words, there exists an $N_0\in\mathbb N$, such that for every $n\geqslant N_0$, we have
\[\inf_{0\leqslant i\leqslant m}\frac{\dim_K\left(\im\left(S^nE_{\hn,p}^{t_i}\rightarrow E_{\hn,np}^{t_i}\right)\right)}{\dim_K(E_{\hn,np}^{t_i})}\geqslant 1-\epsilon\]
for all $i=0,\ldots,m$. Then by the above inequality, we obtain
\begin{eqnarray}\label{equality from appromable}
\dim_K(\mathcal F_{\sp}^{npt_i}(E_{\hn,np}^{(p)}))&\geqslant&\dim_K\left(\im\left(S^nE_{\hn,p}^{t_i}\rightarrow E_{\hn,np}^{t_i}\right)\right)\\
&\geqslant&(1-\epsilon)\dim_K(E_{\hn,np}^{t_i})\nonumber
\end{eqnarray}
for all $i=0,\ldots,m$.

\textbf{Step II. -- An upper bound of integrations.}

We note that
\begin{eqnarray*}
& &np\dim_K(E_{\hn,np}^{(p)})\int_0^{+\infty}t\nu_{\sp,np}^{(p)}(\mathrm dt)=-\int_0^{+\infty}t\mathrm d\dim_K(\mathcal F_{\sp}^tE_{\hn,np}^{(p)})\\
&=&\int_0^{+\infty}\dim_K(\mathcal F_{\sp}^tE_{\hn,np}^{(p)})\mathrm dt\geqslant \sum_{i=1}^mnp(t_i-t_{i-1})\dim_K(\mathcal F_{\sp}^{npt_i}E_{\hn,np}^{(p)})\\
&\geqslant&np(1-\epsilon)\sum_{i=1}^m(t_i-t_{i-1})\dim_K(E_{\hn,np}^{t_i}),
\end{eqnarray*}
where the first inequality is from the decreasing of $\mathcal F_{\sp}$, and the last inequality is from \eqref{equality from appromable} combined with the fact $E_{\hn,np}^{t_i}=\mathcal F_{\sp}^{npt_i}E_{\hn,np}$.

In addition, by applying Abel's summation formula to the last term of the above inequality, we have
\[\dim_K(E_{\hn,np}^{(p)})\int_0^{+\infty}t\nu_{\sp,np}^{(p)}(\mathrm dt)\geqslant(1-\epsilon)\dim_K(E_{\hn,np})\int_0^{+\infty}h_D(t)\nu_{\sp,np}(\mathrm dt).\]

\textbf{Step III. -- Estimates of arithmetic volumes.}

Let $d=\dim(X)$. By Proposition \ref{computation of volume from asymptotic lambda+}, we have
\begin{eqnarray*}
& &\lim_{n\to\infty}\frac{d+1}{(np)^{d+1}/(d+1)!}\dim_K(E_{\hn,np}^{(p)})\int_0^{+\infty}t\nu_{\sp,np}^{(p)}(\mathrm dt)\\
&=&(d+1) \vol(E_\bullet^{(p)})\int_0^{+\infty}t\nu_{\sp}^{(p)}(\mathrm dt)=\widehat{\vol}(E_\bullet^{(p)}).
\end{eqnarray*}
Then we have
\begin{eqnarray*}
\widehat{\vol}(E_\bullet^{(p)})&\geqslant&\lim_{n\to\infty}\frac{d+1}{(np)^{d+1}/(d+1)!}(1-\epsilon)\dim_K(E_{\hn,np})\int_0^{+\infty}h_D\mathrm d \nu_{\sp,np}\\
&=&(d+1)(1-\epsilon)\vol(L)\int_{\mathbb R}h_D\mathrm d \nu_{\sp},
\end{eqnarray*}
where the last equality is from \cite[Chap. IV \S 5 $n^\circ$ 12 Proposition 22]{Bourbaki65}.

We choose a sequence of subdivisions $(D_j)_{j\in\mathbb N}$ satisfying \eqref{zero measure condition of division}, such that the sequence of functions $(h_{D_j}(t))_{j\in\mathbb N}$ converges uniformly to
\[\max\{t,0\}-\max\{t-\wmu_{\max}^{\asy}(\overline E_{\hn,\bullet}),0\}\]
when $j\to\infty$. By definition, the support of $\nu_{\sp}$ is bounded from above by $\wmu_{\max}^{\asy}(\overline E_{\hn,\bullet})$, we obtain
\[\sup_{p\in\mathbb N}\widehat{\vol}(E_\bullet^{(p)})\geqslant (d+1)(1-\epsilon)\vol(L)\int_0^{+\infty}t\nu_{\sp}(\mathrm dt)=(1-\epsilon)\widehat{\vol}(\overline L),\]
where the last equality is from Proposition \ref{same volume over S and S_0} on the equality of arithmetic volumes over $S$ and $S_0$.
\end{proof}
\begin{rema}
For a general graded linear series, if it satisfies the property of Proposition \ref{Fujita approximation of a graded linear series}, we say that this graded linear series satisfies \textit{Fujita approximation}. We would like to emphasize that this notion is different from those in \cite[Definition 3.7]{Chen2018} or \cite[\S 6]{ChenIkoma2020}, though they are somewhat similar.
\end{rema}
\section{Fujita approximation of adelic line bundles}\label{Chap5}
From the Fujita approximation of graded linear series proved in Proposition \ref{Fujita approximation of a graded linear series}, we will prove the arithmetic Fujita approximation of an arithmetically big line bundle with coefficients in $\mathbb Z$, $\mathbb Q$ and $\mathbb R$.

\subsection{Fujita approximation of big adelic line bundles}
First, we will prove the arithmetic Fujita approximation of usual adelic line bundles. We will prove the nef version first, and apply it to prove the ample version. In order to apply Proposition \ref{properties of nef}, we assume the base field is perfect.
\begin{theo}\label{arithmetic Fujita approximation of nef}
Let $S=(K,(\Omega,\mathcal A,\nu),\phi)$ be a proper adelic curve with a perfect $K$, where either $\mathcal A$ is a discrete $\sigma$-algebra or $K$ has a countable sub-field which is dense in $K_w$ for all $w\in\Omega$. Let $X$ be a projective variety with $\dim(X)=d$, and $\overline L$ be an arithmetically big line bundle on $X$ over $S$. Then for every $\epsilon>0$, there exists a birational morphism $v:X'\rightarrow X$ together with a positive integer $p$, such that the decomposition $v^*(\overline L^{\otimes p})=\overline M\otimes \overline G$ satisfies:
\begin{enumerate}
\item[(1)] $\overline M$ has a non-zero small section and $\overline G$ is arithmetically nef;
\item[(2)] $p^{-(d+1)}\widehat{\vol}(\overline G)\geqslant\widehat{\vol}(\overline L)-\epsilon$.
\end{enumerate}
\end{theo}
\begin{proof}
We devide the proof into three steps.

\textbf{Step I. -- Determining the geometric decomposition.}

Let $\pi:X\rightarrow\spec K$ be the structural morphism. For any $p\geqslant1$ such that $E_{\hn,p}^0\neq0$ (see Proposition \ref{contain an ample divisor} for the existence of such $p\in\mathbb N$), let
\[\phi_p:X_p\rightarrow X\]
be the blow-up along the base locus of $E_{\hn,p}^0$, which means
\[X_p=\proj\left(\im\left(\bigoplus_{n\geqslant0}\pi^*(E_{\hn,np}^{(p)})\longrightarrow \bigoplus_{n\geqslant0}L^{\otimes np}\right)\right).\]

Let $G_p=\O_{X_p}(1)$, and $M_p$ be the line bundle on $X_p$ with respect to the exceptional divisor. Let $s$ be a global section of $M_p$ which trivializes $M_p$ outside the exceptional divisor. Then by definition, we have
\[\phi_p^*(L^{\otimes p})=G_p\otimes M_p.\]

The canonical morphism
\[\phi_p^*\pi^*(E_{\hn,p}^0)\rightarrow G_p\]
induces a projective embedding $i_p:X_p\rightarrow\mathbb P\left(E_{\hn,p}^0\right)$ with $i^*\left(\O_{E_{\hn,p}^0}(1)\right)=G_p$. Then the restriction of global sections of $\O_{E_{\hn,p}^0}(n)$ on $X_p$ gives an injective homomorphism
\begin{equation}
\im\left(S^n E_{\hn,p}^0\rightarrow E_{\hn,np}\right)=E_{\hn,np}^{(p)}\rightarrow H^0(X_p,G_p^{\otimes n}),
\end{equation}
where $H^0(X_p,G_p^{\otimes n})$ is considered as a sub-$K$-vector space of $H^0(X_p,\phi_p^*(L^{\otimes p}))$ via the section $s$.

\textbf{Step II. -- Constructing the metrics of line bundles. }

We consider $\overline E^{(p)}_\bullet$ as a normed graded linear series over $S$, which is equipped with the subspace norms of $\overline E_\bullet=\bigoplus\limits_{D\geqslant0}\overline E_D$. For any $n\geqslant 1$ and $w\in\Omega$, let $\|\ndot\|_{n,w}$ be the quotient norm on $G_{p,w}=G_p\times_{\spec K}\spec K_w$ induced by the surjective homomorphism
\[\phi_p^*\pi^*(E_{\hn,np}^{(p)})\rightarrow G_p^{\otimes n},\]
where $E_{\hn,np}^{(p)}$ is equipped with the superior norm $\|\ndot\|_{\sup,w}$ with respect to $w\in\Omega$.

In fact, for any $x\in X_{p,w}(\overline K_w)$ outside the exceptional divisor, it corresponds to a $1$-dimensional quotient of $E_{\hn,p}^0$. For any $n\geqslant1$, this quotient induces a $1$-dimensional quotient $\ell_{n,x}$ of $E_{\hn,np}^{(p)}$.

By Hahn-Banach theorem, every linear form on $\ker(E_{\hn,np}^{(p)}\rightarrow\ell_{n,x})$ can be extended to a linear form on $E_{\hn,np}^{(p)}$. In other words, there exists a section $w\in E_{\hn,np}^{(p)}$, whose image in $G_{p,w}^{\otimes n}(x)$ has norm $\|w\|_{\sup,w}$. As a section of $L_w^{\otimes np}$ on $X_v(\overline K_w)$, we have $\|w(x)\|_w\leqslant\|w\|_{\sup,w}$, where $\|\ndot\|_w$ is the metric on $\overline L$ indexed by $v\in\Omega$. Then for any section $u\in G_{p,w}$ over a neighborhood of $x$, we may assume that $u(x)^{\otimes n}$ equals to the image of $w(x)$ in $G_{p,w}^{\otimes n}(x)$ by dilation, then we have
\[\|u(x)\|_{n,w}=\|w\|_{\sup,w}^{1/n}\geqslant\|w(x)\|_{w}^{1/n}=\|u(x)\otimes s(x)\|_w.\]

Let $\alpha_n=(\|\ndot\|_{n,w})_{w\in\Omega}$ be a family of adelic metrics on $G_{p}$, and we define $(M_p,\beta_n)=\phi_p^*(\overline L)\otimes(G_p,\alpha_n)^\vee$. In this sense, $s$ is an small section in $(M_p,\beta_n)$, then it is an adelic line bundle have a non-zero small section.

\textbf{Step III. -- The relation among arithmetic volumes.}

By the construction in Step II, for any $n\geqslant1$, we have
\[p^{-(d+1)}\widehat{\vol}(G_p,\alpha_n)\geqslant \widehat{\vol}(E_\bullet^{(p)},\alpha_n),\]
where
\[\widehat{\vol}(E_\bullet^{(p)},\alpha_n)=\lim_{m\rightarrow\infty}\frac{\adeg_+(E_{\hn,mp}^{(p)},(\|\ndot\|_{n,\sup,w})_{w\in\Omega})}{m^{d+1}/(d+1)!},\]
and $\|\ndot\|_{n,\sup,w}$ is the superior norm induced by $\alpha_n=(\|\ndot\|_{n,w})_{w\in\Omega}$ on $G_p$. Then by \cite[Corollary A.2]{Chen10}, we have
\[\limsup_{n\to\infty}\widehat{\vol}(E_\bullet^{(p)},\alpha_n)\geqslant\widehat{\vol}(\overline E_\bullet^{(p)}).\]

By Proposition \ref{Fujita approximation of a graded linear series}, for any $\epsilon>0$, there exists a $p\in\mathbb N$, such that
\[\limsup_{n\to\infty}p^{-(d+1)}\widehat{\vol}(G_p,\alpha_n)\geqslant\widehat{\vol}(\overline L)-\frac{\epsilon}{2}.\]
Then there exists an $N_0\in\mathbb N$, such that
\[p^{-(d+1)}\widehat{\vol}(G_p,\alpha_n)-\frac{\epsilon}{2}\geqslant\widehat{\vol}(\overline L)-\frac{\epsilon}{2}\]
for infinitely many $n\geqslant N_0$.

In addition, since $\wmu_{\min}(\overline E_{\hn,np}^{(p)})\geqslant0$ for all $n\geqslant1$, then it has a positive asymptotic minimal slope. Then by Proposition \ref{properties of nef}, we obtain that $(G_p,\alpha_n)$ is arithmetically nef with the fact that $K$ is perfect.

To sum up, we obtain the assertion.
\end{proof}

Next, we obtain the following ample version from Theorem \ref{arithmetic Fujita approximation of nef} directly. The idea is absorbed from \cite[Remark 4.4]{Chen10} and \cite[Corollary 8.5.7]{ChenMoriwaki2022}.
\begin{coro}\label{arithmetic Fujita approximation of ample}
We keep all the notations and conditions in Theorem \ref{arithmetic Fujita approximation of nef}. Then we are able to require $\overline G$ to be arithmetically ample in the term (1) of Theorem \ref{arithmetic Fujita approximation of nef}.
\end{coro}
\begin{proof}
Let $\overline B$ be an arithmetically big line bundle on $X$. Then by \cite[Corollary 6.4.10]{ChenMoriwaki2020}, for $m\in\mathbb N$ sufficiently large, we have $\overline L^{\otimes m}\otimes \overline B^\vee$ is arithmetically big. By Theorem \ref{arithmetic Fujita approximation of nef}, for any $\epsilon>0$, there exists a birational morphism $\phi:X'\rightarrow X$, a $p\in\mathbb N$, and a decomposition
\[\phi^*(\overline L^{\otimes mp}\otimes \overline B^{\otimes (-p)})= \overline G\otimes \overline M\]
with $\overline G$ arithmetically nef and $\overline M$ having a non-zero small section, such that
\[p^{-(d+1)}\widehat{\vol}(\overline G)\geqslant \widehat{\vol}(\overline L^{\otimes m}\otimes \overline B^\vee)-\epsilon.\]
Let $\overline B'$ be an arithmetically ample line bundle on $X'$. Since $\phi^*(\overline B)$ is arithmetically big by \cite[Remark 5.7]{Luo2025}, there exists a $q\in\mathbb N$, such that $\phi^*(\overline B^{\otimes q})\otimes \overline B'^\vee$ has a non-zero small section. We consider
\[\phi^*(\overline L^{\otimes mpq})=(\overline G^{\otimes q}\otimes \overline B')\otimes(\overline M^{\otimes q}\otimes \phi^*(\overline B^{\otimes pq})\otimes \overline B'^\vee).\]
In the above decomposition, $\overline G':=\overline G^{\otimes q}\otimes \overline B'$ is arithmetically ample, $\overline M^{\otimes q}\otimes \phi^*(\overline B^{\otimes pq})\otimes \overline B'^\vee$ has a non-zero small section, and
\[(mpq)^{-(d+1)}\widehat{\vol}(\overline G')\geqslant(mpq)^{-(d+1)}\widehat{\vol}(\overline G^{\otimes q})\geqslant m^{-(d+1)}(\widehat{\vol}(\overline L^{\otimes m}\otimes \overline B^\vee)-\epsilon).\]
By the continuity of the arithmetic volume function (see \cite[Corollary 6.4.10]{ChenMoriwaki2020}), we have the result.
\end{proof}
\subsection{Case of $\mathbb K$-line bundles}
In this part we will prove the arithmetic Fujita approximation is verified for adelic $\mathbb Q$-line bundles and $\mathbb R$-line bundles over adelic curves.
\subsubsection{}
First, we prove the case of adelic $\mathbb Q$-line bundle, which is a quite straightforward corollary from the usual case.
\begin{coro}\label{arithmetic Fujita approximation of Q-ample}
Let $S=(K,(\Omega,\mathcal A,\nu),\phi)$ be a proper adelic curve with a perfect $K$, where either $\mathcal A$ is a discrete $\sigma$-algebra or $K$ has a countable sub-field which is dense in $K_w$ for all $w\in\Omega$. Let $X$ be a projective variety over $\spec K$ with $\dim(X)=d$, and $\overline L$ be an arithmetically big $\mathbb Q$-line bundle on $X$ over $S$. For every $\epsilon>0$, there exists a birational morphism $\phi:X'\rightarrow X$ together with a $p\in\mathbb N^+$, such that the decomposition $\phi^*(\overline L^{\otimes p})=\overline M\otimes \overline G$ satisfies:
\begin{enumerate}
\item[(1)] $\overline M$ and $\overline G$ are adelic line bundles, where $\overline M$ has a non-zero small section and $\overline G$ is arithmetically ample;
\item[(2)] $p^{-(d+1)}\widehat{\vol}(\overline G)\geqslant\widehat{\vol}(\overline L)-\epsilon$.
\end{enumerate}
\end{coro}
\begin{proof}
By definition, there exists an $m\in\mathbb N$, such that $\overline L^{\otimes m}$ is a usual arithmetically big line bundle. Then we obtain the assertion by applying Corollary \ref{arithmetic Fujita approximation of ample} to $\overline L^{\otimes m}$.
\end{proof}
\subsubsection{}
Next, we prove the case of adelic $\mathbb R$-line bundle, and see \cite[\S 6.4.3]{ChenMoriwaki2020} for an introduction to this notion. For this target, we prove an auxiliary lemma first, which allows us to approximate an arithmetically big $\mathbb R$-line bundle by an arithmetically big $\mathbb Q$-line bundle along an effective direction.

Let $\overline{L}_1,\ldots,\overline{L}_n$ be a family of adelic line bundle, and $a_1,\ldots,a_n\in\mathbb R$. We denote by $\sum\limits_{i=1}^na_iL_i$ and $\sum\limits_{i=1}^na_i\overline{L}_i$ the tensor products of line bundles $\bigotimes\limits_{i=1}^nL_i^{\otimes a_i}$ and $\bigotimes\limits_{i=1}^n\overline{L}_i^{\otimes a_i}$ to simplify the notations in the proof below, which are a usual $\mathbb R$-line bundle and an adelic $\mathbb R$-line bundle respectively.

\begin{lemm}\label{density of Q-line bundles}
Let $S=(K,(\Omega,\mathcal A,\nu),\phi)$ be a proper adelic curve, and $X$ be a projective variety over $\spec K$. Let $\overline L$ be an arithmetically big $\mathbb R$-line bundle on $X$. Then for all $\epsilon>0$, there exists an arithmetically big $\mathbb Q$-line bundle $\overline L'$, such that $\overline L-\overline L'$ has a non-zero small section and
\[0<\widehat{\vol}(\overline L)-\widehat{\vol}(\overline L')<\epsilon.\]
\end{lemm}
\begin{proof}
The adelic $\mathbb R$-line bundle $\overline L$ is arithmetically big, so $L$ is a big $\mathbb R$-line bundle. By an equivalent definition \cite[Definition 2.2.21]{LazarsfeldI}, we have $L=\sum\limits_{i=1}^ra_iL_i$, where $a_1,\ldots,a_r>0$ and $L_1,\ldots,L_r$ are usual big line bundles. By definition, $L_1,\ldots,L_r$ are all effective in the geometric sense.

We equip metrics on $L_1,\ldots,L_r$ over $S$, such that $\overline L_1,\ldots,\overline L_r$ have non-zero small sections. Then for any $b_1,\ldots,b_r\in\mathbb R^+$, we have $b_1\overline L_1+\cdots+b_r\overline L_r$ has a non-zero small section.

By the continuity of $\widehat{\vol}(\ndot)$ from \cite[Corollary 6.4.24]{ChenMoriwaki2020}, for all $\epsilon>0$, there exist $\delta_1,\ldots,\delta_r>0$, such that for all $t_i\in(0,\delta_i)$, $i=1,\ldots,r$, we always have
\[0<\widehat{\vol}(\overline L)-\widehat{\vol}\left(\overline L-\sum_{i=1}^rt_i\overline L_i\right)<\epsilon.\]

If there does not exist any arithmetically big $\mathbb Q$-line bundles satisfying the requirements of the lemma, then $\overline L-\sum\limits_{i=1}^rt_i\overline L_i$ can never be a $\mathbb Q$-line bundle for all $t_i\in(0,\delta_i)$, $i=1,\ldots,r$. It means there is no $\mathbb Q$-point in the region
\[(a_1-\delta_1,a_1)\times\cdots\times(a_r-\delta_r,a_r)\subset\mathbb R^{r},\]
which has a strict positive Lebesgue measure. Then it contradicts to the fact that $\mathbb Q^{r}$ is dense in $\mathbb R^{r}$ with respect to the Euclidean topology.

By choosing the constant $\epsilon$ sufficiently small, we pick a family $t_1,\ldots,t_r\in\mathbb R^+$ such that $\overline L-\sum\limits_{i=1}^rt_i\overline L_i$ is an arithmetically big $\mathbb Q$-line bundle. By the above construction, $\sum\limits_{i=1}^rt_i\overline L_i$ has a non-zero small section. Let\[\overline L'=\overline L-\sum\limits_{i=1}^rt_i\overline L_i,\] which satisfies the requirement.
\end{proof}

We prove the Fujita approximation of $\mathbb R$-line bundles via Lemma \ref{density of Q-line bundles}.
\begin{theo}\label{arithmetic Fujita approximation of R-ample}
Let $S=(K,(\Omega,\mathcal A,\nu),\phi)$ be a proper adelic curve with a perfect $K$, where either $\mathcal A$ is a discrete $\sigma$-algebra or $K$ has a countable sub-field which is dense in $K_w$ for all $w\in\Omega$. Let $X$ be a projective variety over $\spec K$ with $\dim(X)=d$, and $\overline L$ be an arithmetically big adelic $\mathbb R$-line bundle on $X$ over $S$. For every $\epsilon>0$, there exists a birational morphism $\phi:X'\rightarrow X$ together with a $p\in\mathbb N^+$, such that the decomposition $\phi^*(\overline L^{\otimes p})=\overline M\otimes \overline G$ satisfies:
\begin{enumerate}
\item[(1)] $\overline M$ is an $\mathbb R$-line bundle having a non-zero small section, and $\overline G$ is an arithmetically ample line bundle;
\item[(2)] $p^{-(d+1)}\widehat{\vol}(\overline G)\geqslant\widehat{\vol}(\overline L)-\epsilon$.
\end{enumerate}
\end{theo}
\begin{proof}
By Lemma \ref{density of Q-line bundles}, for all $\epsilon>0$, there exists an arithmetically big $\mathbb Q$-line bundle $\overline L'$, such that $\overline L-\overline L'$ has a non-zero small section and
\[0<\widehat{\vol}(\overline L)-\widehat{\vol}(\overline L')<\epsilon.\]
Then we apply Corollary \ref{arithmetic Fujita approximation of Q-ample} to $\overline L'$, and we obtain the assertion.
\end{proof}
\section*{Declaration}
\subsection*{Conflict of interest statement}
The author declares that there is no potential conflict of interest.
\subsection*{Ethical statement}
The author guarantees that the manuscript has neither been published elsewhere nor submitted for publication in another journal, and will not be submitted to other journals while being reviewed. The author certify that they have written entirely the manuscript and that its main results are original, and cite all necessary sources. The author will inform immediately to the editors if any error is found in the manuscript after the submission.
\subsection*{The data availability statement}
The data that support the findings of this study are openly available. 
\backmatter

\bibliography{liu}
\bibliographystyle{smfplain}

\end{document}